\documentclass[12pt,reqno]{amsart}
\usepackage[a4paper]{geometry}
\usepackage{amssymb,amsmath,amsthm,enumerate,bbm,array}
\usepackage[%pagebackref,
hypertexnames=false, colorlinks, citecolor=red, linkcolor=red]{hyperref} 
\hypersetup{bookmarksdepth=3}

\DeclareMathOperator{\Dom}{Dom}
\DeclareMathOperator{\Ker}{Ker}
\DeclareMathOperator{\supp}{supp}

\renewcommand{\Re}{\operatorname{Re}}
\renewcommand{\Im}{\operatorname{Im}}

\newcommand{\abs}[1]{\lvert#1\rvert}
\newcommand{\Abs}[1]{\left\lvert#1\right\rvert}
\newcommand{\norm}[1]{\lVert#1\rVert}
\newcommand{\jap}[1]{\langle#1\rangle}

\newcommand{\comp}{{\mathrm{comp}}}

\newcommand{\bbR}{{\mathbb R}}
\newcommand{\bbC}{{\mathbb C}}
\newcommand{\bbZ}{{\mathbb Z}}
\newcommand{\bbP}{{\mathbb P}}

\newcommand{\calF}{\mathcal{F}}
\newcommand{\calL}{\mathcal{L}}
\newcommand{\calM}{\mathcal{M}}
\newcommand{\calD}{\mathcal{D}}

\newcommand{\calR}{\mathcal{R}}

\newcommand{\dd}{\mathrm d}

\numberwithin{equation}{section}

\theoremstyle{plain}
\newtheorem{theorem}{\bf Theorem}[section]
\newtheorem*{theorem*}{Theorem}
\newtheorem{lemma}[theorem]{\bf Lemma}
\newtheorem{proposition}[theorem]{\bf Proposition}
\newtheorem*{proposition*}{\bf Proposition}

\theoremstyle{definition}

\newtheorem*{definition*}{\bf Definition}

\theoremstyle{remark}
\newtheorem*{remark*}{\bf Remark}

\newtheorem*{example*}{\bf Example}

\newcommand{\eps}{\varepsilon}

\makeatletter
\newcommand*\bigcdot{\mathpalette\bigcdot@{.5}}
\newcommand*\bigcdot@[2]{\mathbin{\vcenter{\hbox{\scalebox{#2}{\,\,$\m@th#1\bullet$\,\,}}}}}
\makeatother

\begin{document}

\title[Unbounded integral Hankel operators]{Unbounded integral Hankel operators}

\author{Alexander Pushnitski}
\address{Department of Mathematics, King's College London, Strand, London, WC2R~2LS, U.K.}
\email{alexander.pushnitski@kcl.ac.uk}

\author{Sergei Treil}
\address{Department of Mathematics, Brown University, Providence, RI 02912, USA}
\email{treil@math.brown.edu \textrm{(S.\ Treil)}}
\thanks{ST is partially supported by the NSF grant DMS-2154321}

\date{\today}

\dedicatory{To the memory of Dima Yafaev (1948--2024)}

\begin{abstract}
For a wide class of unbounded integral Hankel operators on the positive half-line, we prove essential self-adjointness on the set of smooth compactly supported functions. 
\end{abstract}

\maketitle

%%%%%%%%%%%%%%%%%%%%%%%%%%%%%%%%%%%%%%
%%%%%%%%%%%%%%%%%%%%%%%%%%%%%%%%%%%%%%
\section{Introduction}\label{sec.z}
%%%%%%%%%%%%%%%%%%%%%%%%%%%%%%%%%%%%%%
%%%%%%%%%%%%%%%%%%%%%%%%%%%%%%%%%%%%%%

\subsection{Foreword}
During 2010--2018, Dima Yafaev worked intensively on the subject of Hankel operators. 
In particular, in \cite{Ya1,Ya2,Ya3,Ya4} he addressed the question of the precise definition of unbounded Hankel operators. In this paper, we continue this line of research by taking the next natural steps in the programme initiated by Dima.

\subsection{Hankel operators: discrete and continuous realisations}
We denote 
\[
\bbZ_+=\{0,1,2,\dots\}\text{ and }\bbR_+=(0,\infty)
\text{ and write }
\ell^2:=\ell^2(\bbZ_+)\text{ and }L^2:=L^2(\bbR_+).
\]
Hankel operators can be defined in two equivalent ways: as infinite matrices acting on $\ell^2$ (discrete realisation) and as integral operators acting on $L^2$ (continuous realisation). 

\emph{Discrete realisation:}
for a sequence $\{h_j\}_{j\in\bbZ_+}$ of complex numbers, we will denote by $\gamma_h$ the (possibly unbounded) operator on $\ell^2$, defined on a suitable set of sequences $a\in\ell^2$ by 
\begin{equation}
(\gamma_h a)_k=\sum_{j=0}^\infty h_{k+j}a_j, \quad k\in\bbZ_+\, .
\label{z9}
\end{equation}
In other words, with respect to the standard basis of $\ell^2$, we can identify $\gamma_h$ with the matrix
\[
\gamma_h=\{h_{k+j}\}_{j,k\in\bbZ_+}.
\]

\emph{Continuous realisation:}
if $h:\bbR_+\to\bbC$ is a function, we will denote by $\Gamma_h$ the (possibly unbounded) operator on $L^2$, defined by 
\begin{equation}
(\Gamma_h f)(t)=\int_0^\infty h(t+s)f(s)\dd s, \quad t\in\bbR_+,
\label{z00}
\end{equation}
on a suitable set of functions $f\in L^2$. The function $h$ is called \emph{the kernel function}. 

The discrete and continuous realisations of Hankel operators are unitarily equivalent. Namely, if we define
\[
\varphi_n(t)=L_n(t)e^{-t/2},\quad t>0,\quad n\geq0,
\]
where $L_n$ are the Laguerre polynomials, then $\{\varphi_n\}_{n=0}^\infty$ is an orthonormal basis in $L^2$ and the matrix of any \emph{bounded} integral Hankel operator in this basis is a Hankel matrix. Moreover, any bounded operator in $L^2$ with a Hankel matrix in this basis is an integral Hankel operator, if $h$ is allowed to be a distribution. For the details, see e.g. \cite[Section~1.8]{Peller} (where a scaled version of $\varphi_n$ is used). 

However, this apparently simple statement hides subtle differences between the discrete and continuous cases. These differences become more pronounced when one considers \emph{unbounded} $\gamma_h$ and $\Gamma_h$. For example, in many cases of interest the kernel function $h(t)$ has a singularity as $t\to0$, while there is no natural analogue of this phenomenon for Hankel matrices. For this and other reasons, natural results for integral Hankel operators are often not directly deducible from the analogous results on Hankel matrices, even though the proof strategies are broadly similar.

The study of unbounded Hankel operators is a relatively young area. Dima Yafaev's contribution has been crucial here, both in the discrete and continuous cases.

\subsection{Positive semi-definite Hankel matrices: the quadratic form of $\gamma_\mu$}
The subject of this paper is the precise definition of unbounded integral Hankel operators $\Gamma_h$. To set the scene, we start by recalling known results for unbounded Hankel matrices $\gamma_h$. We will come back to $\Gamma_h$ in the next section. 

Let $\{h_j\}_{j\in\bbZ_+}$ be a sequence of complex numbers. It is a classical theorem of Hamburger that the corresponding Hankel matrix is positive semi-definite,  i.e. 
\[
\{h_{j+k}\}_{j,k=0}^N \quad\text{is positive semi-definite for all $N\geq0$,}
\]
if and only if $h_j$ are moments of a positive measure $\mu$ on $\bbR$:
\[
h_j=\int_\bbR x^j \dd\mu(x), \quad j\geq0.
\]
In this case we will write $\gamma_\mu$ instead of $\gamma_h$ for the Hankel matrix \eqref{z9}. The quadratic form corresponding to $\gamma_\mu$ can be expressed as
\begin{equation}
q_\mu[a]=\sum_{j,k=0}^\infty h_{j+k}a_j\overline{a_k}=\int_{-\infty}^\infty\Abs{\sum_{j=0}^\infty a_j x^j}^2\dd\mu(x)
\label{z7}
\end{equation}
on a suitable set of elements $a\in\ell^2$.

Widom \cite[Theorem 3.1]{Widom} proved that the form $q_\mu$ is bounded on $\ell^2$, i.e. 
\[
q_\mu[a]\leq C\norm{a}^2_{\ell^2},
\]
if and only if 
\begin{equation}
\supp\mu\subset [-1,1]\quad\text{ with }\quad\mu(\{-1\})=\mu(\{1\})=0
\label{z6}
\end{equation}
and the Carleson condition 
\begin{equation}
\mu((-1,-1+x))+\mu((1-x,1))\leq Cx, \quad \forall x\in(0,1)
\label{z6a}
\end{equation}
is satisfied. 
A classical example is $\dd\mu(x)=\dd x$, which corresponds to $h_j=1/(1+j)$; then the corresponding Hankel matrix is Hilbert's matrix with entries $1/(1+j+k)$. 

In what follows, we will only be interested in the measures satisfying \eqref{z6}; this is equivalent to $h_j\to0$ as $j\to\infty$. We will focus on the case when the Carleson condition fails, i.e. $\mu$ is sufficiently ``concentrated'' near $1$ or $-1$. In this case, the quadratic form $q_\mu$ is unbounded on $\ell^2$, and its definition requires some care. The first step here was made by Yafaev in \cite{Ya3}. Let $\calF\subset\ell^2$ be the set of finite linear combinations of vectors $e_k$ of the canonical basis; in other words, $\calF$ consists of sequences with finitely many non-zero entries.
\begin{theorem}\cite[Theorem~1.2]{Ya3}\label{thm:z0} 
Let $\mu$ satisfy \eqref{z6}. 
\begin{enumerate}[\rm(i)]
\item
The quadratic form $q_\mu$, defined on the domain
\[
\Dom q_\mu=\{a\in\ell^2: q_\mu[a]<\infty\}, 
\]
is closed. 
\item
The set $\calF$ is dense in $\Dom q_\mu$ with respect to the norm induced by $q_\mu$. 
\end{enumerate}
\end{theorem}
Part (i) is simple and standard, while part (ii) requires considerable work. There is a partial converse to this theorem, see \cite[Theorem 1]{Ya4}.

By general principles, there exists a unique self-adjoint  operator $\gamma_\mu$ corresponding to $q_\mu$, and this operator is positive semi-definite. Following Yafaev, we will take $\gamma_\mu$ for the precise definition of the Hankel operator corresponding to the Hankel matrix $\{h_{j+k}\}_{j,k\in\bbZ_+}$.

\medskip

This definition of $\gamma_\mu$ leaves open the following questions:
\begin{itemize}
\item
What is the domain of $\gamma_\mu$?
\item
Is $\gamma_\mu$ essentially self-adjoint on some natural set (e.g. on $\calF$)?
\end{itemize}
We discuss these questions in the rest of this section. 

\subsection{Operator domains: the case $h\in\ell^2$}
Let $\mu$ be a measure of the class \eqref{z6}, let $\{h_j\}_{j\in\bbZ_+}$ be the sequence of its moments and let $\gamma_\mu$ be the Hankel operator defined via the quadratic forms as described above.  Let us first discuss the easier case when $h$ is square-summable. Then for every element $e_n$ of the standard basis of $\ell^2$, the action of $\gamma_\mu$ on $e_n$ is well-defined in an obvious way, 
\begin{equation}
(\gamma_\mu e_n)_k=h_{n+k}, \quad k\geq0, 
\label{z10}
\end{equation}
and the right hand side here is an element of $\ell^2$. From here it easily follows that $\calF\subset\Dom\gamma_h$ and $\gamma_\mu$ can be defined directly by \eqref{z9} on the set $\calF$. On the other hand, for every $a\in \ell^2$, we can define the sequence 
\begin{equation}
(\gamma_\mu a)_k=\sum_{j=1}^\infty h_{k+j}a_j,\quad k\geq0
\label{z0c}
\end{equation}
(by Cauchy-Schwarz the series converges absoltely), but this sequence is not necessarily in $\ell^2$.  Moreover, even if this sequence is in $\ell^2$, it is not obvious whether in this case $a\in\Dom\gamma_\mu$ or whether the action of the operator $\gamma_\mu$ (defined in terms of the quadratic form $q_\mu$ as above) is given by \eqref{z0c}. This is clarified in the following theorem. 
\begin{theorem}\cite{GP1,BS}\label{thm:b4}
Let $\mu$ satisfy \eqref{z6} and let $h_\mu$ be square-summable. 
\begin{enumerate}[\rm(i)]
\item
The operator domain of $\gamma_\mu$ is
\[
\Dom \gamma_\mu=\{a\in\ell^2: \gamma_\mu a\in\ell^2\},
\]
where $\gamma_\mu a$ is understood as the sequence defined by \eqref{z0c}. Moreover, the operator $\gamma_\mu$ acts on $\Dom\gamma_\mu$ according to the formula \eqref{z0c}. 
\item
We have  $\calF\subset\Dom\gamma_\mu$ and $\calF$ is dense in $\Dom \gamma_\mu$ in the graph norm of $\gamma_\mu$
\[
\norm{a}_{\gamma_\mu}:=\left(\norm{a}_{L^2}^2+\norm{\gamma_\mu a}_{L^2}\right)^{1/2}.
\]
In other words, the restriction $\gamma_\mu|_{\calF}$ is essentially self-adjoint. 
\end{enumerate} 
\end{theorem}
Let us put this differently. There are three possible ways to define $\gamma_\mu$:
\begin{itemize}
\item
the ``minimal operator'': define $\gamma_\mu$ by formula \eqref{z0c} on $\calF$ and take the closure;
\item
the ``maximal operator": define  $\gamma_\mu$ by formula \eqref{z0c} on those sequences $a\in\ell^2$ for which $\gamma_\mu a$ belongs to $\ell^2$;
\item
via the quadratic forms: define the form $q_\mu$ on $\calF$, take the closure and consider the corresponding self-adjoint operator.
\end{itemize}
Theorem~\ref{thm:b4} shows that all three definitions lead to the same self-adjoint operator $\gamma_\mu$.

\begin{remark*}
There is a clear analogy between Theorem~\ref{thm:b4} and the deep classical theorem of Kato \cite{Kato} on the Schr\"odinger operator $-\Delta+V$ on $L^2(\bbR^d)$, $d>1$, with the potential $0\leq V\in L^2_{\text{loc}}(\bbR^d)$. For such potential, the expression $-\Delta f+Vf$ belongs to $L^2(\bbR^d)$ for all $f\in C^\infty_{\comp}(\bbR^d)$. Kato's theorem asserts that the operator $-\Delta+V$ is essentially self-adjoint on $C^\infty_{\comp}(\bbR^d)$. 
\end{remark*}

Theorem~\ref{thm:b4} was proved independently and more or less simultaneously in \cite[Theorem~3.8]{BS} and \cite[Theorem~1.1]{GP1} by using very different approaches. Both approaches can be generalised in different directions and both are discussed below.

\subsection{Generalisation to the non-self-adjoint case}
Theorem~\ref{thm:b4} can be regarded as a particular case of \cite[Theorem~1.1]{GP1} which concerns non-self-adjoint Hankel matrices. Let $h\in\ell^2$ be a complex-valued sequence. One can define the operator $\gamma_h$ as in \eqref{z9} on the domain 
\[
\Dom\gamma_h=\{a\in\ell^2: \gamma_h a\in \ell^2\},
\]
where $\gamma_h a$ is understood as the sequence from \eqref{z9} (the series over $j$ converges absolutely by Cauchy-Schwarz). 
\begin{theorem}\cite[Theorem~1.1]{GP1}\label{thm:b4b}
Assume that $h\in\ell^2$.
\begin{enumerate}[\rm(i)]
\item
The above defined operator $\gamma_h$ is closed. 
\item
The set $\calF$ is dense in $\Dom\gamma_h$ in the graph norm of $\gamma_h$. 
\item
The adjoint of $\gamma_h$ is given by $\gamma_h^*=\gamma_{\overline{h}}$. 
\end{enumerate}
\end{theorem}

\subsection{Operator domains: the case $h\notin\ell^2$}
We return to self-adjoint $\gamma_\mu$. 
Let $\mu$ be a measure of the class \eqref{z6}, but now suppose that the sequence of moments $h\notin\ell^2$. This case is more difficult, because one cannot define $\gamma_\mu$ directly even on the elements of the standard basis: the r.h.s. of \eqref{z10} is not in $\ell^2$! This case was considered by Berg and Szwarc in \cite[Section~3]{BS}. 
They make the following simple but remarkable observation:
\begin{equation}
h_k-h_{k+2}=\int_{-1}^1 t^k (1-t^2)\dd\mu(t)=\int_{-1}^1 t^k \dd\nu(t),\quad k\geq0,
\label{z12}
\end{equation}
is the sequence of moments of the measure $\dd\nu(t)=(1-t^2)\dd\mu(t)$. Furthermore, this measure satisfies the Carleson condition:
\begin{align*}
\nu((-1,-1+\eps))&+\nu((1-\eps,1))
=
\int_{-1}^{-1+\eps}(1-x^2)\dd\mu(x)
+
\int_{1-\eps}^{1}(1-x^2)\dd\mu(x)
\\
&\leq
2\eps\int_{-1}^{-1+\eps}\dd\mu(x)
+
2\eps\int_{1-\eps}^{1}\dd\mu(x)
\leq 2\mu((-1,1))\eps,
\end{align*}
and therefore the corresponding Hankel operator $\gamma_\nu$ is bounded. (In fact, $\gamma_\nu$ is trace class, but we don't need this.) In particular, the action of $\gamma_\nu$ on the elements of the standard basis is well-defined. Relation \eqref{z12} can be interpreted as the identity
\[
\gamma_\mu(e_n-e_{n+2})=\gamma_\nu e_n, \quad n\geq0.
\]
Thus, the action of $\gamma_\mu$ on vectors 
\[
v_n:=e_n-e_{n+2}, \quad n\geq0,
\]
is well-defined. It is elementary to see that the set $\calF_0$ of all finite linear combinations of $v_n$ is dense in $\ell^2$. This set gives the correct substitute for $\calF$ in the theorem of Berg and Szwarc below. In order to state it, we denote by $S_*$ the standard backward shift operator on $\ell^2$: 
\[
S_*e_0=0, \quad S_*e_n=e_{n-1} \text{ for }n\geq1.
\] 
One observes that $(I-S_*^2)$ maps $\ell^2$ injectively onto the dense subspace $(I-S_*^2)(\ell^2)$ of $\ell^2$; of course, the inverse $(I-S_*^2)^{-1}$ is well-defined on this subspace.

\begin{theorem}\cite[Theorem~3.2]{BS}\label{thm:b4a}
Let $\mu$ be a measure of the class \eqref{z6}, and let $\dd\nu(x)=(1-x^2)\dd\mu(x)$. 
\begin{enumerate}[\rm(i)]
\item
The domain of $\gamma_\mu$ can be described as 
\[
\Dom \gamma_\mu=\{a\in\ell^2: \gamma_\nu a\in (I-S_*^2)(\ell^2)\}
\]
and for $a\in\Dom \gamma_\mu$ we have 
\[
\gamma_\mu a=(I-S_*^2)^{-1}\gamma_\nu a.
\]
\item
Let  $\calF_0$ be the set of finite linear combinations of vectors $v_n=e_n-e_{n+2}$ for $n\geq0$. Then  $\calF_0\subset\Dom \gamma_\mu$ and $\calF_0$ is dense in $\Dom \gamma_\mu$ in the graph norm of $\gamma_\mu$.  In other words, the restriction $\gamma_\mu|_{\calF_0}$ is essentially self-adjoint.
\end{enumerate}
\end{theorem}

\section{Main results}

Our main results are the analogues of Theorems~\ref{thm:b4}, \ref{thm:b4b} and \ref{thm:b4a} for the continuous realisation of Hankel operators.

\subsection{Integral Hankel operators}
We will be interested in the continuous realisation \eqref{z00} of Hankel operators. 
An operator $\Gamma_h$ is positive semi-definite if and only if $h$ can be represented as the Laplace transform of a positive measure $\mu$ on $\bbR$:
\begin{equation}
h(t)=\int_\bbR e^{-tx}\dd\mu(x), \quad t>0,
\label{z2}
\end{equation}
where $\mu$ is such that integral converges for all $t>0$. 
This fact is a continuous analogue of Hamburger's moment problem; a precise statement and proof in a very general context can be found in Yafaev's work \cite[Theorems 5.1 and 5.3]{Ya2}. 

We will be mainly interested in kernels of the form \eqref{z2} (a minor exception is Theorem~\ref{thm:b1a}) and we will write $h=h_\mu$ in this case. We will write $\Gamma_\mu$ instead of $\Gamma_{h_\mu}$ for the corresponding integral Hankel operator (to be defined below).

Furthemore, we will assume that $h_\mu(t)\to0$ as $t\to\infty$. This corresponds to $\mu$ satisfying the condition 
\begin{equation}
\supp\mu\subset [0,\infty)\quad\text{ and }\quad \mu(\{0\})=0
\label{z14}
\end{equation}
(compare with \eqref{z6}). 
The Hankel operator $\Gamma_\mu$ is bounded if and only if in addition to \eqref{z14}, the measure $\mu$ satisfies the Carleson condition
\begin{equation}
\mu((0,x))\leq Cx, \quad \forall x>0
\label{z3}
\end{equation}
(compare with \eqref{z6a}). 
This condition is equivalent to the estimate
\[
\abs{h_\mu(t)}\leq C/t, \quad t>0.
\]
These facts are mentioned in \cite[page 22]{Widom} without proof; for completeness we give a proof in Appendix. 

As a warm-up, we consider three examples:
\begin{itemize}
\item
Let $\mu$ be the Lebesgue measure restricted to $\bbR_+$. Then $h_\mu(t)=1/t$ and $\Gamma_\mu$ is known as the \emph{Carleman operator}. It is well-known  \cite[Section 10.2]{Peller} that the Carleman operator is bounded but not compact. 
\item
Let $\mu$ be a single point mass at $\alpha>0$, with the total mass $=1$. Then $h_\mu(t)=e^{-\alpha t}$ and $\Gamma_\mu$ is a rank one Hankel operator, which (up to a factor) coincides with the projection onto the function $e^{-\alpha t}$ in $L^2(\bbR_+)$. Of course, if $\mu$ is a finite linear combination of point masses, then $\Gamma_\mu$ is a finite rank operator. 
\item
The simplest example of \emph{unbounded} Hankel operators is furnished by the family 
\begin{equation}
\dd\mu(x)=C_\alpha x^{\alpha-1}\dd x, 
\quad 
C_\alpha=1/\Gamma(\alpha), 
\quad
h_\mu(t)=t^{-\alpha},
\label{b0a}
\end{equation}
where $\alpha>0$, $\alpha\not=1$, and $\Gamma(\alpha)$ is the Gamma-function. 
Yafaev in \cite{Ya1} termed these operators \emph{Quasi-Carleman}  (in fact, he gave this name to a more general family of Hankel operators). It turns out that even in this simple case, the precise definition of $\Gamma_\mu$ requires some effort. 
\end{itemize}

\begin{definition*}
We will write $\mu\in\calM$, if $\mu$ is a positive Borel measure on $\bbR$ satisfying \eqref{z14} such that the Laplace transform of $\mu$ is finite, i.e. 
\begin{equation}
\int_0^\infty e^{-tx}\dd\mu(x)<\infty\quad  \forall t>0.
\label{b3}
\end{equation}
\end{definition*}
For $\mu\in\calM$, the kernel function $h_\mu$ is infinitely differentiable in $t>0$ and $h_\mu(t)\to0$ as $t\to\infty$; however, the rate of decay may be arbitrarily slow. Also $h_\mu(t)$ may have a strong singularity as $t\to0_+$. 

\medskip
For $\mu\in\calM$, we address the following questions:
\begin{itemize}
\item
How to define $\Gamma_\mu$ precisely as a self-adjoint operator on $L^2(\bbR_+)$?
\item
What is the domain of $\Gamma_\mu$?
\item
Is $\Gamma_\mu$ essentially self-adjoint on some natural set of ``nice'' functions?
\end{itemize}
\medskip

%%%%%%%%%%%%%%%%%%%%%%%
\subsection{The quadratic form of $\Gamma_\mu$}
\label{sec.z2}
%%%%%%%%%%%%%%%%%%%%%%%
Below  $C^\infty_\comp(\bbR_+)$ is the set of infinitely smooth compactly supported functions on $\bbR_+$; in particular, the support of $f\in C^\infty_\comp(\bbR_+)$  is separated away from the origin. The set $C^\infty_\comp(\bbR_+)$ plays the role of $\calF\subset\ell^2$ in the continuous case. 

Let $\mu\in\calM$; for $f\in C^\infty_\comp(\bbR_+)$, the quadratic form of $\Gamma_\mu$ can be written as 
\begin{align}
Q_\mu[f]
&=\int_0^\infty \int_0^\infty h_\mu(t+s) f(s)\overline{f(t)}\dd s\,  \dd t
\notag
\\
&=\int_0^\infty \int_0^\infty \left\{\int_0^\infty e^{-tx}e^{-sx}\dd\mu(x)\right\} f(s)\overline{f(t)}\dd s\,  \dd t
\notag
\\
&=\int_0^\infty \abs{\calL f(x)}^2\dd\mu(x), 
\label{z4}
\end{align}
where 
\begin{equation}
\calL f(x)=\int_0^\infty e^{-tx}f(t)\dd t, \quad\quad x>0
\label{b0}
\end{equation}
is the Laplace transform of $f$. The continuous analogue of Theorem~\ref{thm:z0} is the following statement, also due to Yafaev. 

%%%%%%%%%%%%%%%%
\begin{theorem}\cite[Theorem~3.10]{Ya1}\label{thm.z1}
%%%%%%%%%%%%%%%%
Let $\mu\in\calM$.
\begin{enumerate}[\rm(i)] 
\item
The quadratic form $Q_\mu$ defined on the domain 
\[ 
\Dom Q_\mu=\left\{f\in L^2: \int_0^\infty \abs{\calL f(x)}^2\dd\mu(x)<\infty\right\}
\]
is closed. 
\item
The set $C^\infty_\comp(\bbR_+)$ is dense in $\Dom Q_\mu$ with respect to the norm induced by $Q_\mu$. 
\end{enumerate}
\end{theorem}
To be precise,  part (ii) was proved in \cite{Ya1} under the additional assumption 
\begin{equation}
\int_0^\infty (1+x)^{-k}\dd\mu(x)<\infty
\quad 
\text{ for some $k>0$.}
\label{z5}
\end{equation}
In this paper we show that this assumption is not necessary. We give a full proof of Theorem~\ref{thm.z1} in Section~\ref{sec.b}. Our construction has some common elements with \cite{Ya1} but also uses a different idea, borrowed from \cite{BS}. 

Again following Yafaev  \cite{Ya1}, we accept
\begin{definition*}
For $\mu\in\calM$, let $\Gamma_\mu$ be the self-adjoint operator corresponding to the quadratic form $Q_\mu$.
\end{definition*}
Below we discuss the domain of $\Gamma_\mu$ and the essential self-adjointness of $\Gamma_\mu$ on a suitable dense set of functions.

\subsection{Operator domains: $h_\mu$ is square integrable at infinity}
Let us first consider the following particular case: assume that 
\begin{equation}
\int_0^\infty \abs{h_\mu(t+s)}^2\dd s<\infty, \quad \forall t>0.
\label{eq:f1}
\end{equation}
This corresponds to the condition $h\in\ell^2$ in the discrete case. 
Of course, if \eqref{eq:f1} is true for \emph{some} $t>0$, it is also true for \emph{all} $t>0$ (but not necessarily for $t=0$). This condition can be equivalently rewritten in terms of the measure $\mu$:
\[
\int_0^\infty \int_0^\infty \frac{e^{-t(x+y)}}{x+y}\dd\mu(x)\dd\mu(y)<\infty, \quad \forall t>0.
\]

For $f\in L^2$, consider the function 
\begin{equation}
F(t)=\int_0^\infty h_\mu(t+s)f(s)\dd s, \quad t>0;
\label{eq:f2}
\end{equation}
the integral converges absolutely by Cauchy-Schwarz. The function $F$ may or may not be in $L^2$. Moreover, even if $F\in L^2$, it is not obvious whether $f\in\Dom\Gamma_\mu$ or $F=\Gamma_\mu f$.  These points are clarified in the following theorem, which is the continuous analogue of Theorem~\ref{thm:b4}.
\begin{theorem}\label{thm:b1}
Let $\mu\in\calM$ be such that $h_\mu$ is square integrable at infinity \eqref{eq:f1}. 
\begin{enumerate}[\rm(i)]
\item
We have  
\[
\Dom\Gamma_\mu
=
\{f\in L^2: F\in L^2\},
\]
where $F$ is the integral \eqref{eq:f2}. Moreover, if $F\in L^2$, then $\Gamma_\mu f=F$. 
\item
We have $C^\infty_\comp(\bbR_+)\subset\Dom\Gamma_\mu$ and $C^\infty_\comp(\bbR_+)$ is dense in $\Dom\Gamma_\mu$ in the graph norm of $\Gamma_\mu$. In other words, the restriction $\Gamma_\mu|_{C^\infty_\comp(\bbR_+)}$ is essentially self-adjoint.
\end{enumerate}
\end{theorem}

\subsection{Generalisation to the non-self-adjoint case}
Theorem~\ref{thm:b1} can be generalised to the non-self-adjoint case as follows. Let $h$ be a complex-valued measurable function on $\bbR_+$ satisfying \eqref{eq:f1}. As in the self-adjoint case, for every $f\in L^2$ we can define the function 
\[
F(t)=\int_0^\infty h(t+s)f(s)\dd s, \quad t>0,
\]
but $F$ is not necessarily in $L^2$. Let us define the operator $\Gamma_hf=F$ with the domain 
\[
\Dom \Gamma_h=\{f\in L^2: F\in L^2\}.
\]
\begin{theorem}\label{thm:b1a}
Assume that $h$ is integrable at infinity \eqref{eq:f1}. 
\begin{enumerate}[\rm(i)]
\item
The above defined operator $\Gamma_h$ is closed. 
\item
We have $C^\infty_\comp(\bbR_+)\subset \Dom \Gamma_h$ and $C^\infty_\comp(\bbR_+)$ is dense in $\Dom \Gamma_h$ with respect to the graph norm of $\Gamma_h$. 
\item
The adjoint of $\Gamma_h$ is given by $\Gamma_h^*=\Gamma_{\overline{h}}$. 
\end{enumerate}
\end{theorem}

This theorem is a generalisation of \cite[Theorem~3.1]{GP2} and in fact our proof of Theorem~\ref{thm:b1a} uses \cite{GP2}; we will give more details in Section~\ref{sec.c}.

\subsection{Domains: general case}
Let us return to the self-adjoint case and lift the assumption \eqref{eq:f1}. Below we discuss the analogue of Theorem~\ref{thm:b4a}. We recall that $\Gamma_\mu$ is defined as the self-adjoint operator corresponding to the quadratic form $Q_\mu$, see Section~\ref{sec.z2}.

The following example demonstrates that even for $f\in C^\infty_\comp(\bbR_+)$ the function $\Gamma_\mu f$ defined by the integral \eqref{eq:f2} may not be in $L^2$. 
\begin{example*}
For $0<\alpha<1/2$, let $h(t)=t^{-\alpha}$, as in \eqref{b0a}. 
Then for $f\in C^\infty_\comp(\bbR_+)$ we have 
\[
\int_0^\infty h(t+s)f(s)\dd s=\left(\int_0^\infty f(s)\dd s\right)t^{-\alpha}+O(t^{-\alpha-1})
\]
as $t\to\infty$. In particular, if the integral of $f$  is non-zero, then the above function does NOT belong to $L^2$. 
\end{example*}
This example suggests that one should consider functions $f\in C^\infty_\comp(\bbR_+)$ with \emph{zero average}. We denote 
\[
C^\infty_{\comp,0}(\bbR_+)=\left\{f\in C^\infty_\comp(\bbR_+): \int_0^\infty f(t)\dd t=0\right\}.
\]
Equivalently, 
\[
C^\infty_{\comp,0}(\bbR_+)=\{g': g\in C^\infty_\comp(\bbR_+)\}.
\]
(This set plays the role of $\calF_0\subset\ell^2$ in the theorem below.)
Heuristically, we observe that for functions $f\in C^\infty_{\comp,0}(\bbR_+)$, one can integrate by parts in the definition \eqref{z00}. To explain this, let us first consider the derivative 
\[
-h_\mu'(t)=\int_0^\infty e^{-tx}x\dd\mu(x)=h_\nu(t),
\]
with the new measure $\dd\nu(x)=x\dd\mu(x)$. It is clear that $\nu\in\calM$. 
We have 
\begin{equation}
h_\nu(t)=O(1/t) \quad\text{ as } t\to\infty. 
\label{z16}
\end{equation}
Indeed, for $t\geq1$ we have
\begin{align*}
-h'_\mu(t)&=\int_0^\infty e^{-tx}x\dd\mu(x)
\leq \int_0^\infty e^{-tx/2}xe^{-x/2}\dd\mu(x)
\\
&=\frac1t \int_0^\infty e^{-tx/2}(tx)e^{-x/2}\dd\mu(x)
\leq \frac{2e^{-1}}{t}\int_0^\infty e^{-x/2}\dd\mu(x)=\frac{C}{t},
\end{align*}
where we have used the elementary estimate
$e^{-x/2}x\leq 2e^{-1}$
and the assumption that the Laplace transform of $\mu$ is finite. The estimate \eqref{z16} shows that $h_\nu$ is square integrable at infinity, and so $\nu$ satisfies the hypothesis of Theorem~\ref{thm:b1}. 
\begin{remark*}
The situation here is slightly more subtle than in the discrete case. In contrast to the discrete case, $\Gamma_\nu$ is not necessarily bounded. In other words, $h_\nu$ is not necessarily $O(1/t)$ as $t\to0$. 
\end{remark*}
Now if $f\in C^\infty_{\comp,0}(\bbR_+)$, we write $f=g'$ with $g\in C^\infty_{\comp}(\bbR_+)$ and integrate by parts:
\begin{equation}
\int_0^\infty h_\mu(t+s)g'(s)\dd s=\int_0^\infty h_\nu(t+s)g(s)\dd s, 
\label{z17}
\end{equation}
or more succinctly
\[
\Gamma_\mu g'=\Gamma_\nu g.
\]
This relation is key for understanding the definition of $\Gamma_\mu$. Essentially, it reduces the analysis of $\Gamma_\mu$ to the analysis of $\Gamma_\nu$, and the latter is covered by Theorem~\ref{thm:b1}. 

Let us state the continuous analogue of Theorem~\ref{thm:b4a}. 
\begin{theorem}\label{thm:b2}
Let $\mu\in\calM$ and $\dd\nu(x)=x\dd\mu(x)$. 
\begin{enumerate}[\rm (i)]
\item
For $f\in L^2$, denote 
\begin{equation}
F(t)=-\int_0^\infty h_\nu(t+s)f(s)\dd s, \quad t>0.
\label{eq:f11}
\end{equation}
Then $f\in\Dom\Gamma_\mu$ if and only if $F$ is a distributional derivative of a function in $L^2$. Moreover, in this case $\Gamma_\mu f$ is uniquely defined by 
\[
(\Gamma_\mu f)'(t)=F(t), \quad t>0.
\]
\item
We have $C^\infty_{\comp,0}(\bbR_+)\subset\Dom\Gamma_\mu$ and $C^\infty_{\comp,0}(\bbR_+)$ is dense in $\Dom\Gamma_\mu$ in the graph norm of $\Gamma_\mu$. Equivalently, the restriction $\Gamma_\mu|_{C^\infty_{\comp,0}(\bbR_+)}$ is essentially self-adjoint.
\end{enumerate}
\end{theorem}
The proof is given in Section~\ref{sec.f}. 

\subsection{The kernel of $\Gamma_\mu$}
As a ``bonus part'', we give a simple characterisation of the kernel of $\Gamma_\mu$ that follows readily from our construction:
\begin{theorem}\label{prp.z1a}
Let $\mu\in\calM$. 
The kernel of $\Gamma_\mu$ is either trivial or infinite dimensional. 
It is infinite dimensional if and only if $\mu$ is a purely atomic measure supported on a sequence of points $\{a_k\}_{k\in\bbZ_+}\subset \bbR_+$ satisfying the Blaschke condition \begin{equation}
\sum_{k\in\bbZ_+} \frac{a_k}{a_k^2+1}<\infty.
\label{z0a}
\end{equation}
\end{theorem}
The fact that the kernel is either trivial or infinite dimensional is well-known in the context of bounded Hankel operators.

\subsection{Dictionary}
For the reader's guidance, we display the table of correspondences between the discrete and continuous cases. 

\medskip

{\renewcommand{\arraystretch}{1.3}  % Increases row height
\begin{center}
\begin{tabular}{|c|c|}
\hline
Discrete & Continuous
\\
\hline
$\mu(\bbR\setminus(-1,1))=0$ & $\mu(\bbR\setminus(0,\infty))=0$ 
\\
Carleson condition \eqref{z6a} & Carleson condition \eqref{z3}
\\
$\displaystyle h_j=\int_{-1}^1 x^j \dd\mu(x)$ & $\displaystyle h(t)=\int_0^\infty e^{-tx}\dd\mu(x)$
\\
$\gamma_\mu=\{h_{j+k}\}_{j,k=0}^\infty$ & $\Gamma_h$ defined by \eqref{z00}
\\
$q_\mu$ defined by \eqref{z7} & $Q_\mu$ defined by \eqref{z4}
\\
$\calF$ & $C^\infty_\comp(\bbR_+)$
\\
$h\in\ell^2$ & $\int_0^\infty \abs{h(s+t)}^2\dd s<\infty$ $\forall t>0$
\\
$\calF_0$ & $C^\infty_{\comp,0}(\bbR_+)$
\\
$h_k-h_{k+2}$ & $-h'_\mu(t)$
\\
$I-S_*^2$ & $-\dfrac{\dd}{\dd t}$
\\
Theorem~1.N, N$=1,2,3,4$ & Theorem~2.N, N$=1,2,3,4$ 
\\
\hline
\end{tabular}
\end{center}
}

\subsection{Main ideas of proof}
In all Theorems~\ref{thm.z1}-\ref{thm:b2}, the difficult part is (ii). Our proof of Theorem~\ref{thm.z1}(ii) follows Yafaev's general approach of \cite{Ya1}, which is the study of the Laplace transform $\calL$, understood as a closed operator from $L^2(\bbR_+)$ to $L^2(\bbR_+,\mu)$. There are two common elements in the proofs as follows. 

\emph{Commutation relations with the shift semigroup.}
For $\tau>0$, we denote by $S_\tau$ the right shift operator for functions on $\bbR_+$, i.e. 
\[
(S_\tau f)(t)=
\begin{cases}
0, &0<t<\tau,
\\
f(t-\tau), & t\geq \tau.
\end{cases}
\]
We observe that for the Laplace operator $\calL$ we have 
\begin{equation}
(\calL S_\tau f)(x)=e^{-\tau x}(\calL f)(x),
\label{b5}
\end{equation}
and for the Hankel operator $\Gamma_h$ we have
\[
\int_0^\infty h(t+\tau+s)f(s)\dd s
=
\int_\tau^\infty h(t+s)f(s-\tau)\dd s,
\]
that is
\begin{equation}
S_\tau^*\Gamma_h=\Gamma_h S_\tau
\label{z15}
\end{equation}
on suitable domains. These well-known commutation relations are key for our proofs. For example, \eqref{z15} shows that one can approximate any $f\in\Dom\Gamma_h$ by the shifts $S_\tau f$, $\tau\to0$, in the graph norm of $\Gamma_h$.  

\emph{Splitting the measure $\mu$.}
For $\mu\in\calM$ we will write 
\begin{equation}
\mu=\mu_1+\mu_2, \quad \supp\mu_1\subset[0,1], \quad \supp\mu_2\subset[1,\infty)
\label{b4a}
\end{equation}
and handle the measures $\mu_1$ and $\mu_2$ separately. 
Observe that $\mu_1$ is finite and moreover, the measure 
\[
\dd\nu_1(x)=x\dd\mu_1(x)
\]
satisfies the Carleson condition \eqref{z3}. This allows us to use the trick of Berg-Szwarc \cite{BS} to handle the measure $\mu_1$. Handling $\mu_2$ is easier because the corresponding kernel function $h_{\mu_2}(t)$ decays exponentially as $t\to\infty$.

\subsection{The structure of the paper}
In Section~\ref{sec.b} we study the operator of the Laplace transform and prove Theorems~\ref{thm.z1} and \ref{prp.z1a}. In Section~\ref{sec.f} we prove Theorem~\ref{thm:b1}. The proof of Theorem~\ref{thm:b2} is closely aligned with the proof of Theorem~\ref{thm:b1} and is given in Section~\ref{sec.g}. In Section~\ref{sec.c} we consider the non-self-adjoint case and prove Theorem~\ref{thm:b1a}.

%%%%%%%%%%%%%%%%%%%%%%%%%%%%%%%%%%%%%%
%%%%%%%%%%%%%%%%%%%%%%%%%%%%%%%%%%%%%%
\section{The Laplace transform}\label{sec.b}
%%%%%%%%%%%%%%%%%%%%%%%%%%%%%%%%%%%%%%
%%%%%%%%%%%%%%%%%%%%%%%%%%%%%%%%%%%%%%

Here we study the Laplace transform and prove Theorems~\ref{thm.z1} and \ref{prp.z1a}. 

\subsection{Definitions}
Throughout the section, we fix a measure $\mu\in\calM$.
Along with $L^2=L^2(\bbR_+)$, we also consider the space $L^2(\mu)=L^2(\bbR,\mu)$.  We denote the inner product in $L^2$ (resp. in $L^2(\mu)$) by $\jap{f,g}$ (resp. by $\jap{f,g}_\mu$), with the same convention for the norms. Our inner products are linear in the first argument and anti-linear in the second one.

Along with the Laplace transform $\calL$ (see \eqref{b0}), we set 
\begin{align*}
(\calL_\mu f)(t)&=\int_0^\infty e^{-tx}f(x)\dd\mu(x), \quad t>0,\quad f\in L^2(\mu).
\end{align*}
Of course, for $f\in L^2$, the Laplace transform $\calL f$ is not necessarily in $L^2(\mu)$, and likewise for $f\in L^2(\mu)$ the Laplace transform $\calL_\mu f$ is not necessarily in $L^2$. We would like to restrict $\calL$ and $\calL_\mu$ onto appropriate domains to obtain closed linear operators 
\[
L: L^2\to L^2(\mu)\quad\text{ and }\quad L_\mu: L^2(\mu)\to L^2.
\]
(Strictly speaking, here $L$ depends on $\mu$, because the target space depends on $\mu$, but for the purposes of readability we do not reflect this dependence in our notation.)
One can define the \emph{maximal} and \emph{minimal} realisations of $\calL$ and $\calL_\mu$ as follows. 

\emph{Maximal realisations:} let $L$ and $L_\mu$ be the linear operators defined as the restrictions of $\calL$, $\calL_\mu$ onto their \emph{maximal domains}:
\begin{align*}
L f&=\calL f, & f\in \Dom L:&=\{f\in L^2: \calL f\in L^2(\mu)\},
\\
L_\mu f&=\calL_\mu f,  & f\in \Dom L_\mu:&=\{f\in L^2(\mu): \calL_\mu f\in L^2\}.
\end{align*}

\emph{Minimal realisations:}
let $\calD\subset L^2$ and $\calD_\mu\subset L^2(\mu)$ be the sets of bounded compactly supported functions on $\bbR_+$ (this means, in particular, that the supports are bounded away from the origin), and let 
\[
L^{\min}=L|_{\calD},\quad L_\mu^{\min}=L_\mu |_{\calD_\mu}.
\]
Our aim is to show that the minimal and maximal realisations coincide, up to taking closures. We denote by $\overline{A}$ the closure of an operator $A$. 

\begin{remark*}
Of course, instead of the sets $\calD\subset L^2$ and $\calD_\mu\subset L^2(\mu)$ we could take some other suitable dense sets; for example, we could replace $\calD$ by $C^\infty_\comp(\bbR_+)$.
\end{remark*}

\subsection{Main statement and discussion}

\begin{theorem}\label{thm.b1}
Let $\mu\in\calM$. Then:
\begin{enumerate}[\rm (i)]
\item
The operators  $L$ and $L_\mu$ are closed. 
\item
One has 
\[
\overline{L^{\min}}=L\quad\text{ and }\quad\overline{L_\mu^{\min}}=L_\mu.
\]
In other words, $\calD$ is dense in $\Dom L$ with respect to the graph norm of $L$ and $\calD_\mu$ is dense in $\Dom L_\mu$ with respect to the graph norm of $L_\mu$. 
\item
 We have 
\[
L_\mu^*=L \quad\text{ and }\quad L^*=L_\mu.
\]
\item
The kernel of $L_\mu$ is trivial. The kernel of $L$ is either trivial or infinite dimensional. It is infinite dimensional if and only if $\mu$ is a purely atomic measure supported on a sequence of points $\{a_k\}_{k\in\bbZ_+}\subset \bbR_+$ satisfying the Blaschke condition \eqref{z0a}. 
\end{enumerate}
\end{theorem}

Parts (i) and (iv) are very easy. The most subtle part is (ii) (and (iii) follows easily from (ii)). Part (ii) was proved in Section 3 of Yafaev's paper \cite{Ya1} under the additional assumption \eqref{z5}. Here we show that this assumption is not needed. Our proof is based on a different set of ideas from \cite{Ya1} and may be of an independent interest. 

In Section~\ref{sec.b2} we prove that $\calD$ is dense in $\Dom L$ in the graph norm of $L$. Surprisingly, this turns out to be a non-trivial task, and this is the central part of the proof. In Section~\ref{sec.b3}, we prove the rest of Theorem~\ref{thm.b1}.

\subsection{$\calD$ is dense in $\Dom L$ in the graph norm}
\label{sec.b2}

\begin{lemma}\label{lma.b2}
Let $A$ be a bounded self-adjoint positive semi-definite operator in $L^2$. Suppose that for some $u\in L^2$ and all $f\in C^\infty_{\comp}(\bbR_+)$ we have
\[
\jap{u,f'}+\jap{u,Af}=0.
\]
Then $u=0$. 
\end{lemma}
\begin{proof}
Since $A$ is self-adjoint, we can rewrite our assumption as 
\[
\jap{u,f'}+\jap{Au,f}=0, \quad \forall f\in C^\infty_{\comp}(\bbR_+),
\]
which yields $u'=Au$, where $u'$ is the distributional derivative of $u$.
In particular, we find that $u'\in L^2$. By Cauchy-Schwarz, it follows that $u$ is H\"older continuous on $\bbR_+$ and the limit $u(0_+)$ exists. Furthermore, since 
\begin{equation}
\abs{u(t)}^2-\abs{u(0_+)}^2=\int_0^t (u'(s)\overline{u(s)}+u(s)\overline{u'(s)})\dd s
\label{b1c}
\end{equation}
and the integral in the right hand side here extended to $(0,\infty)$ converges absolutely, we see that the limit 
\[
\lim_{t\to\infty}\abs{u(t)}^2
\]
exists. Since $u\in L^2$, this limit equals zero, i.e. we have $u(t)\to0$ as $t\to\infty$. 

Multiplying $u'=Au$ by $\overline{u}$ and integrating, we find
\[
\int_0^\infty u'(t)\overline{u(t)}\dd t=\jap{Au,u}.
\]
Taking the real parts here gives
\[
\int_0^\infty (u'(t)\overline{u(t)}+u(t)\overline{u'(t)})\dd t=2\jap{Au,u}\geq0.
\]
By \eqref{b1c}, the left hand side equals $-\abs{u(0_+)}^2\leq0$, and so $\jap{Au,u}=0$. Since $A$ is positive semi-definite, this yields $Au=0$. Coming back to the equation $u'=Au$, we conclude that $u'=0$ and so $u=\mathrm{const}$. Finally, since $u\in L^2$, we find $u=0$. 
\end{proof}

Let us prove that $C_\comp^\infty(\bbR_+)$ is dense in $\Dom L$ in the graph norm of $L$ (since $C_\comp^\infty(\bbR_+)\subset \calD$, this suffices for our purpose). Our proof proceeds in two steps: we first prove this for finite measures $\mu$ and then extend the construction to general measures. The first step is based on an idea of Berg and Szwarc \cite[Section 1]{BS}. 

\begin{lemma}\label{lma.b2a}
Let $\mu$ be a finite measure on $\bbR_+$ with no point mass at the origin. Then $C_\comp^\infty(\bbR_+)$ is dense in $\Dom L$ with respect to the graph norm of $L$.
\end{lemma}
\begin{proof}
It suffices to prove that $C_{\comp,0}^\infty(\bbR_+)$ is dense. To obtain a contradiction, suppose there exists $u\in\Dom L$ such that 
\begin{equation}
\jap{u,g'}+\jap{L u,L g'}_\mu=0
\label{b1d}
\end{equation}
for all $g\in C_\comp^\infty(\bbR_+)$. Integrating by parts, we find that
\[
(\calL g')(x)=x(\calL g)(x). 
\]
Thus, we can rewrite \eqref{b1d} as 
\begin{equation}
\jap{u,g'}+\jap{\calL u,\calL g}_\nu=0,
\label{b1e}
\end{equation}
where $\dd\nu(x)=x\dd\mu(x)$. Observe that 
\begin{equation}
\int_0^a\dd\nu(x)=\int_0^a x\dd\mu(x)\leq a\int_0^a\dd\mu(x)\leq \mu(\bbR_+)a,
\label{b1f}
\end{equation}
and so $\nu$ satisfies the Carleson condition \eqref{z3}. Thus, the corresponding Hankel operator $\Gamma_\nu$ is bounded and we can rewrite \eqref{b1e} as 
\[
\jap{u,g'}+\jap{u,\Gamma_\nu g}=0
\]
for all $g\in C_\comp^\infty(\bbR_+)$. By Lemma~\ref{lma.b2} (with $A=\Gamma_\nu$), we find that $u=0$, which gives a desired contradiction. 
\end{proof}

\begin{lemma}\label{lma.b2b}
Let $\mu\in\calM$. Then $C_\comp^\infty(\bbR_+)$ is dense in $\Dom L$ with respect to the graph norm of $L$.
\end{lemma}
\begin{proof}
\emph{Step 1: approximation by shifts.}
We use the commutation relation \eqref{b5} of the Laplace transform with the shift $S_\tau$. 
This relation shows that if $f\in\Dom L$ (i.e. if $\calL f\in L^2(\mu)$), then also $S_\tau f\in\Dom L$. Moreover, 
\[
\abs{(L S_\tau f)(x)-(L f)(x)}= \abs{1-e^{-\tau x}}\abs{(L f)(x)},
\]
and so by dominated convergence we have 
\[
\norm{L S_\tau f-L f}_\mu\to0, \quad \tau\to0_+.
\]
Of course, we also have $\norm{S_\tau f-f}\to0$ as $\tau\to0$. Thus, we have the convergence $S_\tau f\to f$ in the graph norm of $L$. 

\emph{Step 2: approximation of the shifted function.}
Let $f\in \Dom L$ and $\tau>0$.  By the previous step, it suffices to approximate $S_\tau f$ by elements of $C_\comp^\infty(\bbR_+)$ in the graph norm. Let us write $\mu=\mu_1+\mu_2$ as in \eqref{b4a}. By our assumption \eqref{b3}, the measure $\mu_1$ is finite. Thus, by Lemma~\ref{lma.b2a}, there exists a sequence $f_n\in C_\comp^\infty(\bbR_+)$ such that 
\begin{equation}
\norm{f_n-f}+\norm{\calL f_n-\calL f}_{\mu_1}\to0,
\label{b4}
\end{equation}
as $n\to\infty$. Denote  $g=S_\tau f$, $g_n=S_\tau f_n$; our aim is to prove that $g_n\to g$ in the graph norm of $L$. Since $g_n\in C_\comp^\infty(\bbR_+)$, this suffices for the proof. 

From \eqref{b4} using the commutation relation \eqref{b5} we find 
\[
\norm{g_n-g}+\norm{\calL g_n-\calL g}_{\mu_1}\to0,
\]
as $n\to\infty$.
It remains to prove that also 
\begin{equation}
\norm{\calL g_n-\calL g}_{\mu_2}\to0
\label{b7}
\end{equation}
as $n\to\infty$. As $g_n$ and $g$ are supported on $[\tau,\infty)$, we find 
\begin{align*}
\abs{(\calL g_n)({x})-(\calL g)({x})}
&\leq
\int_{\tau}^\infty e^{-{x} t}\abs{g_n(t)-g(t)}\dd t
\leq
\norm{g_n-g}\left\{\int_\tau^\infty e^{-2{x} t}\dd t\right\}^{1/2}
\\
&=
\norm{g_n-g}\frac{e^{-\tau{x}}}{\sqrt{2{x}}}.
\end{align*}
Integrating over $\mu_2$ gives 
\begin{align*}
\norm{\calL g_n-\calL g}_{\mu_2}^2
\leq
\norm{g_n-g}^2\int_1^\infty \frac{e^{-2\tau{x}}}{2{x}}\dd\mu_2({x})
\leq 
\frac12\norm{g_n-g}^2\int_0^\infty e^{-2\tau{x}}\dd\mu({x}),
\end{align*}
where the integral in the right hand side is finite by our assumption \eqref{b3}. 
Since $\norm{g_n-g}\to0$, we find that the right hand side converges to zero as $n\to\infty$. This proves \eqref{b7}. The proof of Lemma~\ref{lma.b2b} is complete. 
\end{proof}

\subsection{Proof of Theorem~\ref{thm.b1}}
\label{sec.b3}
\mbox{}

\textbf{Proof of (i):}
This is standard. For example, let us check that $L$ is closed. Suppose $f_n\to f$ in $L^2$ and $L f_n\to F$ in $L^2(\mu)$. From $f_n\to f$ we conclude that $\calL f_n\to \calL f$ pointwise (in fact, uniformly on any interval separated away from the origin). It follows that $F=\calL f$ and so $\calL f\in L^2(\mu)$; thus, $f\in\Dom L$, as required. 

\textbf{Proof of (ii), part $\overline{L^{\min}}=L$:} this follows directly from 
Lemma~\ref{lma.b2b}. The second identity $\overline{L_\mu^{\min}}=L_\mu$ will follow at the next step of the proof by a duality argument.

\textbf{Proof of (iii):}
Here we follow Yafaev's argument of \cite[Lemma 3.2]{Ya1} word for word. 
For $f\in L^2(\mu)$, we have $f\in\Dom (L^{\min})^*$ with $(L^{\min})^*f=F\in L^2$ if and only if 
\[
\jap{f, L^{\min}g}_\mu=\jap{F,g},\quad \forall g\in\calD.
\]
Interchanging the order of integration in the left hand side here by Fubini, 
we rewrite the last identity as 
\[
\int_0^\infty (\calL_\mu f)(t)\overline{g(t)}\dd t
=
\int_0^\infty F(t)\overline{g(t)}\dd t, \quad \forall g\in\calD.
\]
The last relation holds true if and only if $\calL_\mu f=F$, i.e. if and only if $f\in\Dom {L_\mu}$ with ${L_\mu} f=F$. 
This proves 
\begin{equation}
(L^{\min})^*=L_\mu.
\label{b8}
\end{equation}
Observe that for any closable operator $A$ we have $A^*=(\overline{A})^*$. Thus, using that 
$\overline{L^{\min}}=L$, from here we obtain $L^*=L_\mu$. By taking adjoints, we also get $L_\mu^*=L$. 

\textbf{Proof of (ii), part $\overline{L_\mu^{\min}}=L_\mu$:}
By the same argument as \eqref{b8} in the previous step, we find that
\[
(L_\mu^{\min})^*=L.
\]
From here we find
\[
\overline{L_\mu^{\min}}=(L_\mu^{\min})^{**}=L^*=L_\mu,
\]
as required. 

\textbf{Proof of (iv):}
We first note that the Laplace transform $\calL_\mu f$ extends to the right half-plane $\Re t>0$ as an analytic function. Thus, if $\calL_\mu f(t)=0$ for all $t>0$, then it vanishes identically in the right half-plane. It is a well-known classical fact (see e.g. \cite[Chapter II, Theorem 6.3]{Widder}) that this implies that $f=0$; hence the kernel of ${L_\mu}$ is trivial. 

Suppose $L f=0$ for some non-zero $f$; then $\calL f(x)=0$ for $\mu$-a.e. $x>0$. Observe that $\calL f(x)$ is real-analytic in $x>0$ and belongs to the Hardy space in the right half-plane $\Re x>0$. Thus, if $\calL f$ vanishes on $\supp \mu$ for some non-zero $f$, then the support of $\mu$ is the zero set of a Hardy class function, and so it satisfies the Blaschke condition (see e.g. \cite[Section VI]{Koosis}). In this case, the kernel corresponds to the Beurling space of functions in the Hardy class vanishing on $\supp\mu$; this Beurling space is infinite-dimensional. 
The proof of Theorem~\ref{thm.b1} is now complete.
\qed

\subsection{Proof of Theorem~\ref{thm.z1}}
By the definition of the form $Q_\mu$, we have
\begin{equation}
Q_\mu[f]=\norm{Lf}_\mu^2, \quad f\in C^\infty_\comp(\bbR_+).
\label{b9}
\end{equation}
By Theorem~\ref{thm.b1}, the form in the right hand side is closable and the domain of the closure is $\Dom L$. This proves Theorem~\ref{thm.z1}. \qed

\subsection{Proof of Theorem~\ref{prp.z1a}}
By \eqref{b9} we have $\Ker\Gamma_\mu=\Ker L$. Now the required statement follows from Theorem~\ref{thm.b1}(iv). \qed

%%%%%%%%%%%%%%%%%%%%%%%%%%%%%%
%%%%%%%%%%%%%%%%%%%%%%%%%%%%%%
\section{Proof of Theorem~\ref{thm:b1}}
\label{sec.f}
%%%%%%%%%%%%%%%%%%%%%%%%%%%%%%
%%%%%%%%%%%%%%%%%%%%%%%%%%%%%

After an initial step in Lemma~\ref{lma.f1}, we prove Theorem~\ref{thm:b1}(ii); this is the core of the proof. In the last subsection, we prove Theorem~\ref{thm:b1}(i). 

\subsection{Initial step in the proof}
It is convenient to formulate an initial step in the proof of Theorem~\ref{thm:b1} as a lemma.

\begin{lemma}\label{lma.f1}
Assume the hypothesis of Theorem~\ref{thm:b1}. Then $C^\infty_{\comp}(\bbR_+)\subset\Dom\Gamma_\mu$ and for $f\in C^\infty_{\comp}(\bbR_+)$, we have 
\begin{equation}
(\Gamma_\mu f)(t)=\int_0^\infty h_\mu(t+s)f(s)\dd s, \quad t>0.
\label{a00}
\end{equation}
\end{lemma}
We recall that $\Gamma_\mu$ is defined via the quadratic form $Q_\mu$ and so \eqref{a00} is not self-evident. 
\begin{proof}
From \eqref{b9} we find
\[
\Gamma_\mu=L^*L,
\quad
\Dom\Gamma_\mu=\{f\in\Dom L: L f\in\Dom L^*\}.
\]
Let $f\in C^\infty_\comp(\bbR_+)$. We need to check that $Lf\in\Dom L^*$, i.e. $\calL_\mu Lf\in L^2$. By assumption, $f$ is compactly supported in $\bbR_+$; let $a>0$ be such that $\supp f\subset[a,\infty)$. Then we have 
\begin{equation}
\abs{\calL f(x)}\leq \int_{a}^\infty e^{-tx}\abs{f(t)}\dd t\leq Ce^{-ax}, \quad t>0.
\label{a0a}
\end{equation}
It follows that 
\begin{align*}
\abs{(\calL_\mu Lf)(t)}
&\leq C\int_0^\infty e^{-tx}e^{-ax}\dd\mu(x)=Ch_\mu(t+a),
\end{align*}
and the right hand side is in $L^2$ by our assumption \eqref{eq:f1}. 
We conclude that $\calL_\mu Lf\in L^2$. 
Applying Fubini, we find 
\begin{align*}
(\Gamma_\mu f)(t)&=\calL_\mu Lf(t)
=\int_0^\infty e^{-tx}\left\{\int_0^\infty e^{-ts}f(s)\dd s\right\}\dd\mu(x)
=\int_0^\infty h_\mu(t+s)f(s)\dd s,
\end{align*}
which proves \eqref{a00}. 
\end{proof}

We next prove part (ii) of the theorem: $C^\infty_{\comp}(\bbR_+)$ is dense in $\Dom\Gamma_\mu$ in the graph norm of $\Gamma_\mu$. 
An important initial step is to consider the case of finite measures $\mu$. 

\subsection{The case of finite measures $\mu$}\label{sec.f2}
\begin{lemma}\label{thm.a1}
Let $\mu\in\calM$ be a \underline{finite} measure such that $h_\mu$ is square integrable at infinity \eqref{eq:f1}. Then the restriction $\Gamma_\mu|_{C^\infty_{\comp}(\bbR_+)}$ is essentially self-adjoint. 
\end{lemma}
\begin{proof}
In fact, we will prove that the restriction 
\[
\Gamma_\mu^{\min}:=\Gamma_\mu|_{C^\infty_{\comp,0}(\bbR_+)}
\]
is essentially self-adjoint. Obviously, $C^\infty_{\comp,0}(\bbR_+)\subset C^\infty_{\comp}(\bbR_+)$, so this will suffice for our purposes. 

Let $\dd\nu(x)=x\dd\mu(x)$; as in the proof of Lemma~\ref{lma.b2a} (see \eqref{b1f}), we observe that $\nu$ is a Carleson measure and therefore the Hankel operator $\Gamma_\nu$ is bounded. 
Any function in $f\in C^\infty_{\comp,0}(\bbR_+)$ can be written as $f=g'$ with $g\in C^\infty_{\comp}(\bbR_+)$. Using \eqref{a00} and integrating by parts as in \eqref{z17}, we find
\begin{equation}
\Gamma_\mu^{\min}g'=\Gamma_\nu g, \quad g\in C^\infty_\comp(\bbR_+).
\label{z18}
\end{equation}
Thus, we have
\begin{align*}
\jap{\Gamma_\mu^{\min}g',g'}
=
\jap{\Gamma_\nu g,g'}
=
\int_0^\infty (\calL g)(x)\overline{(\calL g')(x)}\dd\nu(x)
=
\int_0^\infty \abs{(\calL g)(x)}^2 x\dd\nu(x)\geq0;
\end{align*}
observe that for $g\in C^\infty_\comp(\bbR_+)$ the Laplace transform $(\calL g)(x)$ is bounded and decays exponentially as $x\to\infty$ and so the integrals above converge absolutely. 
From the above inequality it follows that $\Gamma_\mu^{\min}$ is symmetric and positive semi-definite. Thus, the deficiency indices of $\Gamma_\mu^{\min}$ in the upper and lower half-planes coincide and it suffices to prove that 
\[
\Ker ((\Gamma_\mu^{\min})^*+I)=\{0\}.
\]
Suppose $u\in \Ker ((\Gamma_\mu^{\min})^*+I)$; this means that $u\perp (\Gamma_\mu^{\min}g'+g')$ for any $g\in C^\infty_\comp(\bbR_+)$. By \eqref{z18}, this means
\[
u\perp (\Gamma_\nu g+g'), \quad g\in C^\infty_\comp(\bbR_+).
\]
By Lemma~\ref{lma.b2} with $A=\Gamma_\nu$, it follows that $u=0$, as required. 
\end{proof}

\subsection{Proof of Theorem~\ref{thm:b1}(ii)}\label{sec.f3}
\mbox{}

\emph{Step 1: approximation by shifts.}
We would like to justify and use the commutation relation \eqref{z15} of $\Gamma_\mu=L^*L$ with the shift operator. As in the proof of Lemma~\ref{lma.b2b}, we first note that $S_\tau(\Dom L)\subset\Dom L$ and for $f\in\Dom L$, 
\[
(LS_\tau f)(x)=e^{-\tau x}(Lf)(x), \quad x>0.
\]
Furthermore, for $g\in\Dom L^*$ we have 
\[
L^*\{e^{-\tau x}g(x)\}(t)=(L^*g)(t+\tau)=(S_\tau^*L^*g)(t).
\]
We conclude that 
\[
S_\tau(\Dom \Gamma_\mu)\subset\Dom\Gamma_\mu
\]
and 
\begin{equation}
\Gamma_\mu S_\tau=S_\tau^*\Gamma_\mu
\label{eq:f3}
\end{equation}
on $\Dom\Gamma_\mu$. 

Now suppose $f\in\Dom\Gamma_\mu$; denote $F=\Gamma_\mu f=L^*L f\in L^2$. Clearly, we have 
\[
\norm{S_\tau f-f}+\norm{S_\tau^*F-F}\to0, \quad \tau\to0.
\]
By \eqref{eq:f3}, this rewrites as
\[
\norm{S_\tau f-f}+\norm{\Gamma_\mu(S_\tau f-f)}\to0, \quad \tau\to0.
\]
In other words, $S_\tau f\to f$ in the graph norm of $\Gamma_\mu$. 

\emph{Step 2: splitting the measure.}
We write $\mu=\mu_1+\mu_2$ as in \eqref{b4a}. Along with $\Gamma_\mu$, we consider  the Hankel operators $\Gamma_{\mu_1}$ and $\Gamma_{\mu_2}$.  
We would like to write
\begin{equation}
\Gamma_\mu=\Gamma_{\mu_1}+\Gamma_{\mu_2},
\label{eq:f4}
\end{equation}
but we need to be careful about the domains here. Let us first prove that 
\[
\Dom\Gamma_\mu\subset\Dom\Gamma_{\mu_1}
\quad\text{ and }\quad
\Dom\Gamma_\mu\subset\Dom\Gamma_{\mu_2}.
\]
Let $f\in\Dom\Gamma_\mu$ and denote $F=\Gamma_\mu f=L_\mu Lf\in L^2$. We can write
\[
F(t)=F_1(t)+F_2(t), 
\]
with 
\[
F_1(t)=(\calL_{\mu_1}\calL f)(t)
\quad\text{ and }\quad
F_2(t)=(\calL_{\mu_2}\calL f)(t).
\]
Clearly, both $F_1$ and $F_2$ are in $C^\infty(\bbR_+)$. 
We need to check that both $F_1$ and $F_2$ are in $L^2$. Informally, we must check that there are no significant cancellations in the sum $F=F_1+F_2$; we will do this by checking that $F_1$ is regular near $t=0$ and  $F_2$ decays exponentially fast as $t\to\infty$. 

Since $f\in\Dom\Gamma_\mu$, we have, in particular,
\[
\int_0^\infty \abs{\calL f(x)}^2\dd\mu(x)<\infty.
\]
Using this, we find that
\[
\abs{F_1(t)}
=\Abs{\int_0^1 e^{-tx}(\calL f)(x)\dd\mu(x)}
\leq \int_0^1 \abs{(\calL f)(x)}\dd\mu(x)
\leq C, 
\]
and so $F_1$ is uniformly bounded on $\bbR_+$. Since 
\[
F=F_1+F_2\in L^2(\bbR_+)\subset L^2(0,1),
\] 
we find that both $F_1$ and $F_2$ are in $L^2(0,1)$.

Next, for $t\geq1$ we have 
\begin{align}
\abs{F_2(t)}
&=\Abs{\int_1^\infty e^{-tx}(\calL f)(x)\dd\mu(x)}
\leq e^{-t/2}\int_1^\infty e^{-x/2}\abs{(\calL f)(x)}\dd\mu(x)
\notag
\\
&\leq 
e^{-t/2}
\left\{\int_1^\infty e^{-x}\dd\mu(x)\int_1^\infty \abs{(\calL f)(x)}^2\dd\mu(x)\right\}^{1/2}
=
Ce^{-t/2},
\label{eq:x1}
\end{align}
and so $F_2$ decays exponentially as $t\to\infty$. Since $F=F_1+F_2\in L^2(1,\infty)$, we find that both $F_1$ and $F_2$ are in $L^2(1,\infty)$. 

We conclude that both $F_1$ and $F_2$ are in $L^2(\bbR_+)$, as required. In particular, the relation \eqref{eq:f4} is valid on $\Dom\Gamma_\mu$.

\emph{Step 3: approximation of the shifted function.}
We have 
\[
h_\mu=h_{\mu_1}+h_{\mu_2}.
\]
Since
\begin{align*}
h_{\mu_2}(t) = \int_1^\infty e^{-tx} \dd\mu(x) \le e^{-t/2} \int_1^\infty e^{-x/2} \dd\mu(x) = 
e^{-t/2} \calL\mu (1/2)
\end{align*}
for $t\geq1$, we see that the kernel function $h_{\mu_2}$ decays exponentially at infinity, and by 
our hypothesis, $h_\mu$ is square-integrable at infinity. It follows that $h_{\mu_1}$ is 
square-integrable at infinity. 

Let $f\in\Dom\Gamma_\mu$ and $\tau>0$. By Step 1 of the proof, it suffices to approximate $S_\tau f$ by functions in $C^\infty_{\comp}(\bbR_+)$ in the graph norm of $\Gamma_\mu$. By Step~2 of the proof, $f\in\Dom\Gamma_{\mu_1}$. By Lemma~\ref{thm.a1} applied to $\mu_1$, the restriction $\Gamma_{\mu_1}|_{C^\infty_{\comp}(\bbR_+)}$ is essentially self-adjoint (here we use that $h_{\mu_1}$ is square-integrable at infinity). Thus, there exists a sequence $f_n\in C^\infty_{\comp}(\bbR_+)$ such that 
\begin{equation}
\norm{f_n-f}+\norm{\Gamma_{\mu_1}(f_n-f)}\to0, \quad n\to\infty.
\label{eq:f6}
\end{equation}
Denote $g=S_\tau f$, $g_n=S_\tau f_n$. By Step~2 of the proof, we have $g\in\Dom\Gamma_{\mu_1}\cap\Dom\Gamma_{\mu_2}$ and we can write 
\[
\Gamma_\mu g=\Gamma_{\mu_1}g+\Gamma_{\mu_2}g.
\]
Of course, we can also write this for $g_n$ in place of $g$ because $g_n\in C^\infty_{\comp}(\bbR_+)$ (here we use Lemma~\ref{lma.f1}). 
Thus, we have
\begin{equation}
\norm{g_n-g}+\norm{\Gamma_\mu(g_n-g)}
\leq
\norm{g_n-g}+\norm{\Gamma_{\mu_1}(g_n-g)}+\norm{\Gamma_{\mu_2}(g_n-g)},
\label{eq:f7}
\end{equation}
and our aim is to prove that the right hand side goes to zero as $n\to\infty$.

Using the commutation relation \eqref{eq:f3}, we find
\begin{align*}
\norm{g_n-g}&+\norm{\Gamma_{\mu_1}(g_n-g)}
=
\norm{S_\tau(f_n-f)}+\norm{\Gamma_{\mu_1}S_\tau(f_n-f)}
\\
&=\norm{S_\tau(f_n-f)}+\norm{S_\tau^*\Gamma_{\mu_1}(f_n-f)}
\leq 
\norm{f_n-f}+\norm{\Gamma_{\mu_1}(f_n-f)}\to0
\end{align*}
as $n\to\infty$ by \eqref{eq:f6}. Returning to \eqref{eq:f7}, we see that it remains to prove that 
\begin{equation}
\norm{\Gamma_{\mu_2}(g_n-g)}\to0, \quad n\to\infty.
\label{eq:f8}
\end{equation}
As $g_n$ and $g$ are supported on $[\tau,\infty)$, we find, as in the proof of Lemma~\ref{lma.b2b},
\[
\abs{(\calL g_n)({x})-(\calL g)({x})}
\leq 
\norm{g_n-g}\frac{e^{-\tau{x}}}{\sqrt{2{x}}}
\]
and so 
\begin{align*}
\abs{\Gamma_{\mu_2}(g_n-g)(t)}
&\leq 
\norm{g_n-g}\int_1^\infty \frac{e^{-\tau{x}}}{\sqrt{2{x}}}e^{-t{x}}\dd\mu({x})
\\
&\leq 
e^{-t}
\norm{g_n-g}\int_1^\infty e^{-\tau{x}}\dd\mu({x})
=
C(\tau)\norm{g_n-g}e^{-t}.
\end{align*}
Thus we find that 
\[
\norm{\Gamma_{\mu_2}(g_n-g)}
\leq 
C(\tau)\norm{g_n-g}\left\{\int_0^\infty e^{-2t}\dd t\right\}^{1/2}
=
\frac{C(\tau)}{\sqrt{2}}\norm{g_n-g}
\to0
\]
as $n\to\infty$. This yields \eqref{eq:f8}.
The proof of Theorem~\ref{thm:b2}(ii) is complete. 
\qed

\subsection{Proof of Theorem~\ref{thm:b1}(i)}
First we make a calculation. Let $f\in L^2$ and let $F$ be defined by the integral \eqref{eq:f2}. Then for any $g\in C^\infty_{\comp}(\bbR_+)$, applying Fubini at the first step and Lemma~\ref{lma.f1} at the second step, we find 
\begin{align}
\int_0^\infty F(t)\overline{g(t)}\dd t&=\int_0^\infty \left\{\int_0^\infty h_\mu(t+s)f(s)\dd s\right\}\overline{g(t)}\dd t
\notag
\\
&=\int_0^\infty f(s)\left\{\int_0^\infty h_\mu(t+s)\overline{g(t)}\dd t\right\}\dd s
=\jap{f,\Gamma_\mu g}.
\label{eq:f14}
\end{align}
Next, we have the following chain of equivalences. 
\begin{itemize}
\item
The inclusion $F\in L^2$ is equivalent to 
\[
\Abs{\int_0^\infty F(t)\overline{g(t)}\dd t}\leq C\norm{g}, \quad \forall g\in C^\infty_\comp(\bbR_+).
\]
\item
By \eqref{eq:f14}, the latter estimate is equivalent to 
\[
\abs{\jap{f,\Gamma_\mu g}}\leq C\norm{g}, \quad \forall g\in C^\infty_\comp(\bbR_+).
\]
\item
Since $C^\infty_\comp(\bbR_+)$ is dense in $\Dom\Gamma_\mu$ in the graph norm of $\Gamma_\mu$, the last estimate is equivalent to 
\[
\abs{\jap{f,\Gamma_\mu g}}\leq C\norm{g}, \quad \forall g\in\Dom\Gamma_\mu.
\]
\item
Since $\Gamma_\mu$ is self-adjoint, the last estimate is equivalent to $f\in\Dom\Gamma_\mu$. 
\end{itemize}
We have proved that $\Dom\Gamma_\mu=\{f\in L^2: F\in L^2\}$. 

Finally, if $f\in\Dom\Gamma_\mu$, then by \eqref{eq:f14} we find 
\[
\jap{F,g}=\jap{f,\Gamma_\mu g}=\jap{\Gamma_\mu f,g},
\]
and so $\Gamma_\mu f=F$. The proof of Theorem~\ref{thm:b1}(i) is complete. \qed

%%%%%%%%%%%%%%%%%%%%%%%%%%%%%%
%%%%%%%%%%%%%%%%%%%%%%%%%%%%%%
\section{Proof of Theorem~\ref{thm:b2}}
\label{sec.g}
%%%%%%%%%%%%%%%%%%%%%%%%%%%%%%
%%%%%%%%%%%%%%%%%%%%%%%%%%%%%

\subsection{Initial step in the proof of Theorem~\ref{thm:b2}}
Below is the analogue of Lemma~\ref{lma.f1} in the context of the proof of Theorem~\ref{thm:b2}. We recall that $\dd\nu(x)=x\dd\mu(x)$. We also recall that any function $f\in C^\infty_{\comp,0}(\bbR_+)$ can be written as $f=g'$ with $g\in C^\infty_{\comp}(\bbR_+)$.
\begin{lemma}\label{lma.f2}
Assume the hypothesis of Theorem~\ref{thm:b2}. 
Then $C^\infty_{\comp,0}(\bbR_+)\subset\Dom\Gamma_\mu$ and the action of $\Gamma_\mu$ on $C^\infty_{\comp,0}(\bbR_+)$ is given by 
\begin{equation}
\Gamma_\mu g'=\Gamma_\nu g, \quad g\in C^\infty_{\comp}(\bbR_+).
\label{eq:f10}
\end{equation}
\end{lemma}
\begin{proof}
This is a slight modification of Lemma~\ref{lma.f1}.
Let $g\in C^\infty_\comp(\bbR_+)$; let us show that $g'\in\Dom\Gamma_\mu$. Since $\Gamma_\mu=L^*L$, we need to check that $Lg'\in\Dom L^*$, i.e. $\calL_\mu Lg'\in L^2$. 
Integrating by parts, we find that
\[
(\calL g')(x)=x(\calL g)(x). 
\]
Let $\supp g\subset[a,\infty)$; using \eqref{a0a}, we find 
\begin{align*}
\abs{(\calL_\mu Lg')(t)}
&\leq C\int_0^\infty e^{-tx}e^{-ax}x\dd\mu(x)
\\
&=\frac{C}{(t+\frac{a}{2})}\int_0^\infty e^{-(t+\frac{a}{2})x}(t+\tfrac{a}2)xe^{-ax/2}\dd\mu(x)
\leq
\frac{Ce^{-1}}{t+\frac{a}2}\int_0^\infty e^{-ax/2}\dd\mu(x),
\end{align*}
because $xe^{-x}\leq e^{-1}$ for $x>0$. The integral in the right hand side here is finite by our assumption on $\mu$. We conclude that $\calL_\mu Lg'\in L^2$. Moreover, recalling that
\[
-h_\mu'(t)=\int_0^\infty e^{-tx}x\dd\mu(x)=h_\nu(t)
\]
and applying Fubini, we find 
\begin{align*}
(\Gamma_\mu g')(t)&=(\calL_\mu Lg')(t)
=\int_0^\infty e^{-tx}\left\{\int_0^\infty e^{-ts}g(s)\dd s\right\}x\dd\mu(x)
\\
&=-\int_0^\infty h_\mu'(t+s)g(s)\dd s=\int_0^\infty h_\nu(t+s)g(s)\dd s=(\Gamma_\nu g)(t),
\end{align*}
where we have used Lemma~\ref{lma.f1} at the last step (we can do this since $h_\nu$ is square integrable at infinity). This proves \eqref{eq:f10}. 
\end{proof}

\subsection{The case of finite measures $\mu$}
Below is the analogue of Lemma~\ref{thm.a1}. 

\begin{lemma}\label{thm.a1x}
Let $\mu\in\calM$ be a \underline{finite} measure. Then the restriction $\Gamma_\mu|_{C^\infty_{\comp,0}(\bbR_+)}$ is essentially self-adjoint. 
\end{lemma}
\begin{proof}
This is a slight modification of the proof of Lemma~\ref{thm.a1}. 
Denote 
\[
\Gamma_\mu^{\min}:=\Gamma_\mu|_{C^\infty_{\comp,0}(\bbR_+)}.
\]
By Lemma~\ref{lma.f2}, we have 
\[
\Gamma_\mu^{\min}g'=\Gamma_\nu g, \quad g\in C^\infty_\comp(\bbR_+), 
\]
which is exactly \eqref{z18} in the proof of Lemma~\ref{thm.a1}.
Then, continuing exactly as in the proof of Lemma~\ref{thm.a1}, we find that 
$\Gamma_\mu^{\min}$ is symmetric and positive semi-definite and check that its deficiency indices 
vanish. 
\end{proof}

\subsection{Proof of Theorem~\ref{thm:b2}(ii)}
This is a slight modification of the proof of Theorem~\ref{thm:b1}(ii) in Section~\ref{sec.f3}. Steps 1 and 2 of the proof require no modifications. In Step 3, we simply replace $C^\infty_{\comp}(\bbR_+)$ by $C^\infty_{\comp,0}(\bbR_+)$ everywhere and use Lemmas~\ref{lma.f2} and \ref{thm.a1x} instead of Lemmas~\ref{lma.f1} and \ref{thm.a1}.  

\subsection{Proof of Theorem~\ref{thm:b2}(i)}
First we make a calculation. Let $f\in L^2$ and let $F$ be defined by the integral \eqref{eq:f11}. For any $g\in C^\infty_{\comp}(\bbR_+)$, applying Fubini, we get 
\begin{align*}
\int_0^\infty F(t)\overline{g(t)}\dd t
&=
-\int_0^\infty\left\{\int_0^\infty h_\nu(t+s)f(s)\dd s\right\}\overline{g(t)}\dd t
\notag
\\
&=-\int_0^\infty f(s)\left\{\int_0^\infty h_\nu(t+s)\overline{g(t)}\dd t\right\}\dd s.
\end{align*}
Applying Lemma~\ref{lma.f1} to $\Gamma_\nu$, we get for the integral in brackets
\[
\int_0^\infty h_\nu(t+s)g(t)\dd t=\Gamma_\nu g.
\]
Finally, by Lemma~\ref{lma.f2} we have $\Gamma_\nu g=\Gamma_\mu g'$. Putting this together, we obtain 
\begin{equation}
\int_0^\infty F(t)\overline{g(t)}\dd t
=-\jap{f,\Gamma_\mu g'}. 
\label{eq:f13x}
\end{equation}
Next, we have the following chain of equivalences. 

\begin{itemize}
\item
$F$ is a distributional derivative of a function in $L^2$ if and only if the estimate
\[
\Abs{\int_0^\infty F(t)\overline{g(t)}\dd t}\leq C\norm{g'}
\quad\forall g\in C^\infty_{\comp}(\bbR_+)
\]
holds true. 
\item
By \eqref{eq:f13x}, the latter estimate is equivalent to 
\[
\abs{\jap{f,\Gamma_\mu g'}}\leq C\norm{g'}, 
\quad \forall g\in C^\infty_{\comp}(\bbR_+).
\]
\item
Since $C^\infty_{\comp,0}(\bbR_+)$ is dense in $\Dom\Gamma_\mu$ in the graph norm of $\Gamma_\mu$, the latter estimate is equivalent to 
\[
\abs{\jap{f,\Gamma_\mu v}}\leq C\norm{v}, 
\quad \forall v\in\Dom\Gamma_\mu.
\]
\item
Since $\Gamma_\mu$ is self-adjoint, the latter estimate is equivalent to the inclusion $f\in\Dom\Gamma_\mu$.
\end{itemize}
We have proved that $\Dom\Gamma_\mu$ is as described in the hypothesis of the theorem. 
Finally, if $f\in\Dom\Gamma_\mu$ and $F=G'$ with $G\in L^2$, then by \eqref{eq:f13x} we find
\[
\jap{G,g'}=\jap{f,\Gamma_\mu g'}=\jap{\Gamma_\mu f,g'},
\quad\forall g\in C^\infty_{\comp}(\bbR_+).
\]
and so $G=\Gamma_\mu f$, as claimed. 
The proof of Theorem~\ref{thm:b2}(i) is complete. 
\qed

%%%%%%%%%%%%%%%%%%%%%%%%%%%%%%%%%%%%%%
%%%%%%%%%%%%%%%%%%%%%%%%%%%%%%%%%%%%%%
\section{The non-self-adjoint case: proof of Theorem~\ref{thm:b1a}}\label{sec.c}
%%%%%%%%%%%%%%%%%%%%%%%%%%%%%%%%%%%%%%
%%%%%%%%%%%%%%%%%%%%%%%%%%%%%%%%%%%%%%

\subsection{Unbounded Hankel operators on the Hardy class}
We start by recalling a theorem from \cite{GP2} about unbounded Hankel operators on the Hardy class $H^2(\bbR)$ of the upper half-plane. We will use this theorem in the proof of Theorem~\ref{thm:b1a}. 

We denote by $\Phi f=\widehat{f}$ the (unitary) Fourier transform on $L^2(\bbR)$, 
\[
\widehat{f}(\xi)=\frac1{\sqrt{2\pi}}\int_{-\infty}^\infty f(x)e^{-ix\xi}\dd x. 
\]
We will regard the Hardy class $H^2(\bbR)$ as a subspace of $L^2(\bbR)$, 
\[
H^2(\bbR)=\{f\in L^2(\bbR): \supp\widehat f\subset [0,\infty)\}.
\]
Let $\bbP$ be the orthogonal projection from $L^2(\bbR)$ onto $H^2(\bbR)$. 

We start with \emph{bounded} Hankel operators in $H^2(\bbR)$. For a \emph{symbol} $u\in L^\infty(\bbR)\cap H^2(\bbR)$, we define the Hankel operator $H_u$ in $H^2(\bbR)$ by 
\begin{equation}
H_u f=\bbP(u\cdot \widetilde{f}), \quad f\in H^2(\bbR), \quad \widetilde{f}(\xi)=f(-\xi). 
\label{c2x}
\end{equation}
The boundedness of $H_u$ is evident from the condition $u\in L^\infty$. Regarding the Fourier transform $\Phi$ as a unitary operator from $H^2(\bbR)$ onto $L^2(\bbR_+)$, we see that 
\begin{equation}
\Phi H_u \Phi^*=\Gamma_h,
\quad
h=\frac1{\sqrt{2\pi}}\widehat u.
\label{c2}
\end{equation}
Let us now discuss \emph{unbounded} operators $H_u$ with symbols $u\in H^2(\bbR)$. 
Fix $u\in H^2(\bbR)$ and define the operator $H_u$ by \eqref{c2x} on the domain 
\begin{equation}
\Dom H_u=\{f\in H^2(\bbR): \bbP(u\cdot \widetilde{f})\in H^2(\bbR)\}.
\label{c1}
\end{equation}
Here the condition $\bbP(u\cdot \widetilde{f})\in H^2(\bbR)$ should be understood as follows: for the function $g=u\cdot\widetilde{f}\in L^1(\bbR)$, the Fourier transform $\widehat{g}$, restricted onto $\bbR_+$, belongs to $L^2(\bbR)$. 

In the theorem below, the dense set of ``nice'' functions in $H^2(\bbR)$ is chosen to be the set $\calR$ of rational functions. More precisely, $\calR$ consists of functions of the form $p(x)/q(x)$, where $p$ and $q$ are polynomials with $\deg p<\deg q$ and all zeros of $q$ are located in the open lower half-plane $\Im z<0$. 

\begin{theorem}\cite[Theorem~3.1]{GP2}\label{thm:c1}
Let $u\in H^2(\bbR)$ and let $H_u$ be the Hankel operator with the domain \eqref{c1}. Then:
\begin{enumerate}[\rm(i)]
\item
$H_u$ is closed; 
\item
the set $\calR$ of rational functions is dense in $\Dom H_u$ in the graph norm of $H_u$.
\end{enumerate}
\end{theorem}
\begin{remark*}
Theorem~3.1 of \cite{GP2} is stated for the anti-linear version of Hankel operators, viz. 
\[
f\mapsto\bbP(u\cdot f^{\#}), \quad f^\#(\xi)=\overline{f(-\xi)}.
\]
However, this version differs from $H_u$ only by a complex conjugation, and trivial calculations show that the theorem stated above is equivalent to \cite[Theorem~3.1]{GP2}. 
\end{remark*}

\subsection{Mapping to $L^2(\bbR_+)$}
Using the unitary equivalence \eqref{c2}, it is easy to translate Theorem~\ref{thm:c1} into the language of Hankel operators $\Gamma_h$ on $L^2(\bbR_+)$. Clearly, condition $u\in H^2(\bbR)$ on the symbol is equivalent to the condition $h\in L^2(\bbR)$ on the kernel function. Let us identify the image of rational functions $\calR\subset H^2(\bbR)$ under the Fourier transform. Representing rational functions as sums of partial fractions, we see that $\Phi(\calR)$ can be characterised as the set of all finite linear combinations 
\begin{equation}
\sum_n p_n(t)e^{-\alpha_n t},
\label{z0b}
\end{equation}
where $p_n$ is a polynomial and $\Re\alpha_n>0$. 

Now we can rewrite Theorem~\ref{thm:c1} as follows.

\begin{theorem}\label{thm:c2}
Let $h\in L^2$ and let $\Gamma_h$ be the Hankel operator defined by the integral \eqref{z00} on the domain 
\[
\Dom \Gamma_h=\{f\in L^2: \Gamma_h f\in L^2\}.
\]
Then:
\begin{enumerate}[\rm(i)]
\item
$\Gamma_h$ is closed; 
\item
the set $\Phi(\calR)$ is dense in $\Dom \Gamma_h$ in the graph norm of $\Gamma_h$.
\end{enumerate}
\end{theorem}

In order to derive Theorem~\ref{thm:b1a} from Theorem~\ref{thm:c2}, we need to 
\begin{itemize}
\item
Replace condition $h\in L^2$ by condition \eqref{eq:f1};
\item
Replace the set $\Phi(\calR)$ by the set $C^\infty_\comp(\bbR_+)$. 
\end{itemize}

\subsection{Continuity of $\Gamma_h f$}
We begin the proof of Theorem~\ref{thm:b1a}. In order to make the arguments below cleaner, we make the following remark: for any $f\in\Dom\Gamma_h$, the function $\Gamma_h f$ is continuous on $\bbR_+$. Indeed, using the shift operator $S_t$, we can write
\begin{equation}
(\Gamma_h f)(t)=\jap{S_{t}^*h,\overline{f}}, \quad t>0.
\label{c3}
\end{equation}
By our assumption, $S_t^* h\in L^2$ for any $t>0$. Therefore, the function $t\mapsto S_t^* h$ is continuous in the norm of $L^2$ for any $t>0$. Therefore by \eqref{c3}, the function $(\Gamma_h f)(t)$ is continuous in $t>0$.

\subsection{Proof of Theorem~\ref{thm:b1a}(i): $\Gamma_h$ is closed}
As usual, the closedness of $\Gamma_h$ is a simple exercise. Suppose $f_n\in\Dom\Gamma_h$ is a sequence such that $f_n\to f$ in $L^2$ and $\Gamma_h f_n\to F$ in $L^2$. From \eqref{c3} we find that $(\Gamma_h f_n)(t)\to (\Gamma_h f)(t)$ for all $t>0$ (in fact, uniformly on any interval separated away from the origin). It follows that $\Gamma_h f=F$ and so $\Gamma_h f\in L^2$; thus, $f\in\Dom\Gamma_h$ as required. 

\subsection{Approximation by shifts}
As in the proof of Theorem~\ref{thm:b1}(ii), we find that
\[
S_\tau(\Dom \Gamma_h)\subset\Dom\Gamma_h
\]
and 
\[
\Gamma_h S_\tau=S_\tau^*\Gamma_h. 
\]
Now let $f\in\Dom\Gamma_h$ and $F=\Gamma_h f$. Again, as in the proof of Theorem~\ref{thm:b1}(ii), we find that
\[
\norm{S_\tau f-f}+\norm{\Gamma_h(S_\tau f-f)}\to0, \quad \tau\to0, 
\]
and so $S_\tau f\to f$ in the graph norm of $\Gamma_h$.

\subsection{Using Theorem~\ref{thm:c2}}
Let $f\in\Dom\Gamma_h$ and $\tau>0$; by the previous step, it suffices to approximate $S_\tau f$ in the graph norm. Observe that 
\[
\Gamma_h S_\tau f=\Gamma_{h_\tau} f, 
\]
where $h_\tau(t)=h(t+\tau)$. By our assumption, $h_\tau\in L^2$ and so we can apply Theorem~\ref{thm:c2} to the operator $\Gamma_{h_\tau}$. This shows that there exists a sequence $f_n\in \Phi(\calR)$ such that $f_n\to f$ in the graph norm of $\Gamma_{h_\tau}$: 
\[
\norm{f_n-f}+\norm{\Gamma_{h_\tau}(f_n-f)}\to0, \quad n\to\infty.
\]
Then we have 
\[
\norm{S_\tau(f_n-f)}+\norm{\Gamma_h S_\tau(f_n-f)}
=
\norm{S_\tau(f_n-f)}+\norm{\Gamma_{h_\tau}(f_n-f)}
\to0
\]
as $n\to\infty$ and so $S_\tau f_n\to S_\tau f$ in the graph norm of $\Gamma_h$. 

At this stage of the proof, we have established that the set 
\[
\{S_\tau f: f\in \Phi(\calR), \  \tau>0\},
\]
is dense in the graph norm of $\Gamma_h$. In order to complete the proof of Theorem~\ref{thm:b1a}(ii), we need to approximate each element $S_\tau f$, $f\in \Phi(\calR)$, by functions in $C^\infty_\comp(\bbR_+)$ in the graph norm of $\Gamma_h$. 

\subsection{Approximation by functions in $C^\infty_\comp(\bbR_+)$}
Let $f\in L^2(\bbR_+)\cap L^1(\bbR_+)$; we claim that the estimate
\begin{equation}
\norm{\Gamma_h S_\tau f}\leq \norm{S_\tau^* h}\norm{f}_{L^1}
\label{c4}
\end{equation}
holds true. This is simply a particular case of Young's inequality (a convolution with an $L^1$-function is a bounded operator on $L^2$)  written in an unfamiliar notation. In any case, it is easy to give a direct proof. 
We have 
\begin{align*}
\abs{(\Gamma_h S_\tau f)(t)}&\leq \int_\tau^\infty \abs{h(t+s)}\abs{f(s-\tau)}\dd s
=\int_0^\infty \abs{h(t+s+\tau)}\abs{f(s)}\dd s,
\end{align*}
and therefore 
\begin{align*}
\norm{\Gamma_h S_\tau f}^2
&\leq
\int_0^\infty\left\{
\int_0^\infty \abs{h(t+s_1+\tau)}\abs{f(s_1)}\dd s_1
\int_0^\infty \abs{h(t+s_2+\tau)}\abs{f(s_2)}\dd s_2
\right\}\dd t
\\
&=
\int_0^\infty\int_0^\infty
\left\{\int_0^\infty \abs{h(t+s_1+\tau)h(t+s_2+\tau)}\dd t\right\}
\abs{f(s_1)f(s_2)}\dd s_1\dd s_2. 
\end{align*}
For the integral over $t$ by Cauchy-Schwarz we have 
\[
\int_0^\infty \abs{h(t+s_1+\tau)h(t+s_2+\tau)}\dd t
\leq
\norm{S_\tau^* h}^2.
\]
Putting this all together yields \eqref{c4}. 

Now suppose $f\in\Phi(\bbR)$; then $f$ is a finite sum of the form \eqref{z0b}, and in particular $f\in L^1\cap L^2$. Clearly, we can find a sequence $f_n\in C^\infty_\comp(\bbR_+)$ such that 
\[
\norm{f_n-f}_{L^2}+\norm{f_n-f}_{L^1}\to 0, \quad n\to\infty.
\]
Then 
\[
\norm{S_\tau (f_n-f)}+\norm{\Gamma_h S_\tau(f_n-f)}
\leq 
\norm{f_n-f}+\norm{S_\tau^* h}\norm{f_n-f}_{L^1}\to0
\]
as $n\to\infty$, and so $S_\tau f_n\to S_\tau f$ in the graph norm of $\Gamma_h$. This completes the proof of Theorem~\ref{thm:b1a}(ii).

\subsection{Proof of Theorem~\ref{thm:b1a}(iii)}
For $f,g\in C^\infty_{\comp}(\bbR_+)$ by Fubini we find
\[
\jap{\Gamma_h f, g}=\jap{f,\Gamma_{\overline{h}}g}.
\]
Since by the previous step $C^\infty_{\comp}(\bbR_+)$ is dense both in $\Dom\Gamma_h$ and in $\Dom\Gamma_{\overline{h}}$, from here we obtain that $\Gamma_h^*=\Gamma_{\overline{h}}$. \qed

\appendix

\section{}

\subsection{Boundedness of $\Gamma_\mu$}
The statement below can be regarded as folklore; part (i)$\Leftrightarrow$(iii) is mentioned in \cite[p.22]{Widom} without proof.

\begin{proposition}\label{prp.app}
Let $\mu\in\calM$ and let $h_\mu$ be the Laplace transform of $\mu$. 
Then  the following are equivalent:
\begin{enumerate}[\rm (i)]
\item
The measure $\mu$ is Carleson, i.e. $\mu(0,a)\leq Ca$, $a>0$;
\item
The kernel function satisfies 
$h_\mu(t)\leq C/t$, $t>0$; 
\item
The Hankel operator $\Gamma_\mu$ is bounded. 
\end{enumerate}
\end{proposition}
\begin{proof}
(i)$\Rightarrow$(ii):
Integrating by parts, we find
\[
h_\mu(t)
=t\int_0^\infty e^{-t x}\mu((0,x))\dd x
\leq Ct\int_0^\infty e^{-t x}x\dd x
\leq C/t.
\]
(ii)$\Rightarrow$(iii): Since the Carleman operator (corresponding to $h_\mu(t)=1/t$) is bounded with the norm $\pi$, we find 
\[
\jap{\Gamma_\mu f,f}
\leq C \int_0^\infty \int_0^\infty \frac{\abs{f(t)}\abs{f(s)}}{t+s}\dd t\, \dd s\leq C\pi \norm{f}^2.
\]
(iii)$\Rightarrow$(i): Assume (iii).
Take $f_\eps(t)=e^{-\eps t}$, then $\norm{f_\eps}^2=1/2\eps$, and 
\begin{align*}
\norm{\Gamma_\mu}
&\geq2\eps\jap{\Gamma_\mu f_\eps,f_\eps}
=2\eps\int_0^\infty \abs{\calL f_\eps(x)}^2\dd\mu(x)
=2\eps\int_0^\infty (\eps+x)^{-2}\dd\mu(x)
\\
&\geq2\eps\int_0^a (\eps+x)^{-2}\dd\mu(x)
\geq
2\eps(\eps+a)^{-2}\mu((0,a)).
\end{align*}
Now taking $\eps=a$, we get $\mu((0,a))\leq 2\norm{\Gamma_\mu}a$.
\end{proof}

\end{document}